\documentclass[10pt]{article}
\usepackage{amsfonts}
\usepackage{amsmath}
\usepackage{amssymb}
\usepackage{amsthm}
\usepackage[alphabetic]{amsrefs}
\usepackage{calc}
\usepackage{graphicx}
\usepackage[hidelinks]{hyperref}
\usepackage[left=1in,top=1in,right=1in,foot=1in]{geometry}
\usepackage{todonotes}
\usepackage{microtype}
\usepackage{tikz-cd}

\usepackage[cal=boondox]{mathalfa}

\newcommand{\pair}[2]{\left\langle #1 , #2\right\rangle}

\DeclareMathOperator{\Sym}{Sym}

\DeclareMathOperator{\coker}{coker}

\def\into{\mathrel{\hookrightarrow}}
\def\onto{\mathrel{\twoheadrightarrow}}

\newcommand{\sets}[2]{\left\{#1\,\middle|\,#2\right\}}
\newcommand{\genrel}[2]{\left\langle #1\,\middle|\,#2\right\rangle}

\newcommand{\rquotient}{\backslash}

\DeclareMathOperator{\End}{End}

\DeclareMathOperator{\Mat}{Mat}

\newcommand{\id}{\mathrm{id}}

\DeclareMathOperator{\Ind}{Ind}

\newcommand{\Rep}{\mathbf{Rep}}
\newcommand{\Mod}{\mathbf{Mod}}
\newcommand{\Coh}{\mathrm{Coh}}
\newcommand{\QCoh}{\mathrm{QCoh}}

\DeclareMathOperator{\Spec}{Spec}
\DeclareMathOperator{\relSpec}{\underline{Spec}}

\newcommand{\Pp}{\mathbb{P}}

\newcommand{\Bb}{\mathcal{B}}
\newcommand{\Fl}{\mathcal{F}\ell}

\newcommand{\pt}{\mathrm{pt}}

\newcommand{\HH}{\mathbf{H}}

\newcommand{\Ff}{\mathcal{F}}
\newcommand{\Gg}{\mathcal{G}}

\DeclareMathOperator{\supp}{supp}

\newcommand{\N}{\mathbb{N}}
\newcommand{\Z}{\mathbb{Z}}
\newcommand{\C}{\mathbb{C}}

\newcommand{\F}{\mathbb{F}}
\newcommand{\Oo}{\mathcal{O}}

\newcommand{\Ad}{\mathrm{Ad}}

\newcommand{\g}{\mathfrak{g}}

\newcommand{\bB}{\mathfrak{b}}

\newcommand{\nN}{\mathfrak{n}}

\newcommand{\G}{\mathbf{G}}
\newcommand{\Haff}{H_{\mathrm{aff}}}

\newcommand{\cc}{\mathbf{c}}

\newcommand{\PGL}{\mathrm{PGL}}
\newcommand{\SL}{\mathrm{SL}}

\newcommand{\Gm}{{\mathbb{G}_\mathrm{m}}}

\newcommand{\Sn}{\mathfrak{S}}

\newcommand{\St}{\mathrm{St}}
\newcommand{\triv}{\mathrm{triv}}

\newcommand{\bq}{\mathbf{q}}

\newtheorem{theorem}{Theorem}

\newtheorem{prop}{Proposition}
\newtheorem{lem}{Lemma}

\theoremstyle{definition}
\newtheorem{dfn}{Definition}
\theoremstyle{remark}
\newtheorem{ex}{Example}
\newtheorem{rem}{Remark}

\newcommand{\Waff}{\tilde{W}}

\newcommand{\CHaff}{\mathcal{H}_\mathrm{aff}}

\newcommand{\E}{\mathcal{E}}

\newcommand{\Perv}{\mathrm{Perv}}
\newcommand{\DGCoh}{\mathrm{DGCoh}}

\newcommand{\J}{\mathcal{J}_0}
\newcommand{\JA}{\mathcal{J}_{0\mathcal{A}}}

\newcommand{\tee}{\mathcal{t}}

\newcommand{\heart}{\heartsuit}

\newcommand{\sS}{\mathrm{SingSupp}}

\newcommand{\Sing}{\mathrm{Sing}}
\newcommand{\Nn}{\tilde{\mathcal{N}}}

\newcommand{\BbG}{\Bb/G}
\newcommand{\ptG}{\pt/G}
\newcommand{\ptGGm}{\pt/G\times\Gm}
\newcommand{\BbGGm}{\Bb/G\times\Gm}

\newcommand{\GGm}{G\times\Gm}

\newcommand{\StGGm}{\mathrm{St}/G\times\Gm}

\DeclareMathOperator{\KD}{\mathrm{KD}}

\providecommand{\keywords}[1]{\textbf{\textit{Keywords---}} #1}

\title{A coherent categorification of the based ring of the lowest two-sided cell}
\date{\today}
\author{Stefan Dawydiak \thanks{Mathematical Institute, Universit\"{a}t Bonn, Bonn 53111 Germany; email \texttt{dawydiak@math.uni-bonn.de}}}

\begin{document}
\maketitle
\begin{abstract}
We give a partial coherent categorification  of $J_0$, the based ring of the lowest two sided cell of an affine 
Weyl group, equipped with a monoidal functor from the category of coherent sheaves on the 
derived Steinberg variety. We show that our categorification acts on natural coherent categorifications
of the Iwahori invariants of the Schwartz space of the basic affine space.
In low rank cases, we construct complexes that lift the basis elements $t_w$ of $J_0$ and their structure
constants.
\end{abstract}
\keywords{Asymptotic Hecke algebra, Iwahori-Hecke algebra, flag variety, Steinberg variety}
\section{Introduction}
Let $\Waff$ be an affine Weyl group. Its group algebra $\C[\Waff]$ is deformed by the affine Hecke algebra 
$\Haff=H(\Waff)$ of $\Waff$. In turn,
Lusztig defined the asymptotic Hecke algebra $J$, a based ring with basis $t_w$, $w\in\Waff$
and structure constants determined from certain ``leading terms" of the structure constants of 
$\Haff$. Further, he provided a morphism of algebras $\phi\colon\Haff\into J\otimes_\Z\Z[\bq^{\pm 1/2}]$
and showed it was an algebra after a mild completion. Thus $\Haff$ can be viewed as a subalgebra of $J$,
and $J$ can be viewed as a subalgebra of a completion $\CHaff$ of $\Haff$. While the morphism $\phi$ is an 
essential part of Lusztig's exploration of $J$, until recently there have been few compelling reasons to adopt
the perspective of $J\otimes_\Z\Z[\bq^{\pm 1/2}]$ as a subalgebra of $\CHaff$.

The algebra $\Haff$ appears in 
many areas of mathematics in many guises, but one of the most prominent relates to the representation theory
of $p$-adic groups. Let $F$ be a local non-archimedean field and $q$ be the cardinality of the residue field of $F$. Let $G^\vee$ be a connected reductive algebraic group defined and split over $F$, with Langlands dual group $G$ taken over $\C$. For the purposes of harmonic analysis on $G^\vee(F)$, it is natural to consider $\Haff$, very much an algebraic object, as a subalgebra of the larger, analytically-characterized Harish-Chandra Schwartz algebra $\mathcal{C}(G^\vee(F))^I$. 

In \cite{BK}, Braverman and Kazhdan gave an interpretation of $J$ in terms of harmonic analysis, casting $J$
as an algebraic version of $\mathcal{C}(G^\vee(F))^I$ (they also defined a ring $\mathcal{J}$ doing the same for 
the full algebra $\mathcal{C}(G^\vee(F))$) by defining a map $J\to\mathcal{C}(G^\vee)^I$. In \cite{D2},
the author showed that this morphism was essentially the specialization of $\phi^{-1}$ for $\bq=q$, and in particular was injective. In \cite{BKK}, Bezrukavnikov-Karpov-Krylov showed that this map was an isomorphism;
the author did so via another approach for all but finitely-many cells of exceptional groups in \cite{Rigid}.

Lusztig gave a categorification of $J$ in \cite{LusCat} in terms of perverse sheaves on the affine flag
variety. In the spirit of the definition of $J$ as a ring, the underlying category is again $\Perv(\Fl)$,
but with monoidal structure given by \emph{truncated convolution} as opposed to convolution.
In this paper we provide a categorification of a large direct summand $J_0$ of $J$ that is compatible
with the perspective of \cite{BK}. Namely, we obtain a natural categorification of the action of $J_0$
on the unitary principal series, and produce a completely new category whose $K$-theory is $J_0$, as opposed
a new monoidal structure.

The main result of the present paper is
%
\begin{theorem}
Let $\Bb$ be the derived zero section of the Springer resolution $\tilde{\mathcal{N}}\to\mathcal{N}$. Then
\begin{enumerate}
\item 
the category
\[
\J:=D^b\Coh(\BbG\times_{\ptG}\BbG)_{!}
\]
is a triangulated subcategory of $\Coh(\BbG\times_{\ptG}\BbG)$,
has a monoidal structure given by convolution, and admits a natural monoidal functor
\[
\Coh(\StGGm)\to \Coh(\BbGGm\times_{\ptGGm}\BbGGm)_{!}
\]
such that 
\item
the induced morphism
\[
\Haff\to K_0(\J)\to K_{0}(\Bb^\heart/G\times\Gm\times_{\ptGGm}\Bb^\heart/G\times\Gm)
\]
is conjugate to $\phi_0$;
\item
In the special case when $G$ has universal cover equal to $\SL_2$ or $\SL_3$, there exists a family of objects 
$\{\tee_w\}_{w\in\mathbf{c}_0}$ in $\J$, such that, if $t_wt_x=\sum_{z}\gamma_{w,x,z^{-1}}t_z$ in $J_0$, then
\[
\tee_w\star\tee_x=\bigoplus_z\tee_z^{\oplus\gamma_{w,x,z^{-1}}}
\]
in $\J$ and the image in $ K_{0}(\Bb^\heart/G\times_{\ptG}\Bb^\heart/G)$ of the class $[\tee_w]$ under the 
above morphism is $[t_w]$.
\item
The category $\J$ acts on $\Coh(\Bb/T)$ and on $\Coh(\BbG\times_{\ptG}\BbG)$;
\end{enumerate}
\end{theorem}
%
%

\begin{proof}
The four propositions below each prove one statement of the theorem.
\end{proof}
The algebra $J$ is very close to being a direct 
sum of matrix algebras. In type $A$, this is the main result of Xi's monograph \cite{XiAMemoir}, and Bezrukavnikov-Ostrik in \cite{BO} showed this up to central extensions in general. (It is however now 
known \cite{BDD} \cite{QX} that this partial result is sharp; the central extensions do in fact appear in general.)

Therefore the last two items are particularly relevant: $J$ is most interesting as a based algebra admitting a 
morphism from $\Haff$, and can be quite simple in isolation. Our categorification
captures the failure of $\phi_0$ to be surjective, as explained in Remark \ref{rem singsupp reflects phi0 not surjective} and the discussion preceding it. We hope to remove the very restrictive current hypothesis on
item 3 in a future version of this paper; see remark \ref{rem pairing failure} for an explanation of why 
it is currently necessary.

We would be remiss to not point out that, while the algebra $J_0$ is a quotient of $K_0(\J)$ by Proposition 
\ref{prop K-theory morphisms conjugate}, the two are in fact not equal. 
After completing this paper, we learned of \cite{Propp}, which succeeds in categorifying each summand of $J$, in 
particular $J_0$, in a way that also recovers $\phi_0$. The decategorification procedure of \textit{op. cit.}
is more sophisticated than simply taking Grothendieck groups.
\subsection{Acknowledgements}
The author thanks Alexander Braverman for introducing him to the notion of of singular support and its consequences for convolution, Dylan Butson and Kostya Tolmachov for many patient and helpful conversations about the 
basics of derived algebraic geometry and the coherent Hecke category, and Andy Ramirez-Cot\'{e} for a helpful 
discussion. The author thanks Roman Bezrukavnikov for pointing out an overreach in an earlier version of this
paper. This research was supported by NSERC.

\section{Functions and algebras}
In this section we will recall the various algebras whose categorifications we will discuss in Section
\ref{section sheaves and categories}. There is no new material in this section, although we could not find
a recollection of all the relationships below in one place in the existing literature.
\subsection{The affine Hecke algebra}
Let $G$ be a connected simply-connected reductive group defined over $\C$ with Borel subgroup $B$, maximal torus $T\subset B$,
and classical flag variety $\Bb^\heart=G/B$. Let $X^*$ be the character lattice of $T$, and 
$\Waff=W\ltimes X^*$, where $W$ is the finite Weyl group of $G$. Let $H=\Haff$ be the corresponding affine Hecke 
algebra over $\mathcal{A}=\Z[\bq^{1/2},\bq^{-1/2}]$ with standard basis $\{T_w\}_{w\in\Waff}$. The 
multiplication in this, the Coxeter presentation, of $H$
is determined by $T_wT_{w'}=T_{ww'}$ when $\ell(ww')=\ell(w)+\ell(w')$ and the quadratic relation
$(T_s+1)(T_s-\bq)=0$ for all $s\in S$, where $S\subset\Waff$ is the set of simple reflections. The 
Coxeter presentation is well-suited to studying the action of $H$ on admissible representations and 
the constructible categorification of $H$.

There is a second presentation of $H$, due to Bernstein (and Bernstein-Zelevinskii in type $A$), which
appears naturally in coherent descriptions of $H$, both on the level of $K$-theory and the level of categories.
\begin{dfn}
The \emph{Bernstein presentation} of $H$ is the presentation with basis $\{T_w\theta_\lambda\}_{w\in W,\lambda\in X_*}$, where
\begin{itemize}
\item
For $w\in W$, $T_{w}$ is the same basis element as in the Coxeter presentation.
\item
If $\lambda\in\tilde{W}$ is an antidominant character, and hence corresponds to a \emph{geometrically} dominant 
cocharacter in the sense of \cite{CG}, then 
\[
\theta_{\lambda}=\bq^{-\frac{\ell(\lambda)}{2}}T_{\lambda}.
\]
\item
If $\lambda\in\tilde{W}$ is a dominant character, and hence corresponds to a
\emph{geometrically} antidominant cocharacter, then
\[
\theta_{\lambda}=\bq^{\frac{\ell(\lambda)}{2}}T_{\lambda}^{-1}.
\]
\end{itemize}
The sets $\{\theta_\lambda,\theta_\lambda T_{s_0}\}_{\lambda\in X_*}$ and $\{\theta_\lambda, T_s\theta_\lambda\}_{\lambda\in X_*}$ are each $\mathcal{A}$-bases. The relations are as follows:
\begin{itemize}
\item
The same quadratic relation for $T_{s_0}$;
\item
For any cocharacters $\lambda,\lambda'$, we have
$\theta_\lambda\theta_{\lambda'}=\theta_{\lambda+\lambda'}$.
\item
The Bernstein relation
\[
\theta_{\frac{\alpha}{2}}T_{s_\alpha^\vee}-T_{s_\alpha^\vee}\theta_{-\frac{\alpha}{2}}=(q-1)\frac{\theta_{\frac{\alpha}{2}}-\theta_{-\frac{\alpha}{2}}}{1-\theta_{-\alpha}}
=
(q-1)\theta_{\frac{\alpha}{2}}.
\]
\end{itemize}
\end{dfn}
\begin{ex}
Let $G^\vee=\PGL_2$, so that $\Waff=\Sn_2\ltimes X$, where $X$ is the character lattice of $G=\SL_2$. Write $s_0$ for the finite simple reflection, and $s_1$ for the affine simple reflection in $\Waff$. Then 
the generators of the Bernstein subalgebra are as follows:
\begin{enumerate}
\item 
If $\lambda=-n\alpha^\vee=(s_1s_0)^n\in\tilde{W}$,  and hence corresponds to a \emph{geometrically} dominant 
cocharacter in the sense of \cite{CG}, then 
\[
\theta_{(s_1s_0)^n}=\theta_{-n\alpha^\vee}= \bq^{-n}T_{(s_1s_0)^n}.
\]
\item
If $\lambda=n\alpha^\vee=(s_0s_1)^n\in\tilde{W}$, and hence corresponds to a
\emph{geometrically} antidominant cocharacter, then
\[
\theta_{(s_0s_1)^n}=\theta_{n\alpha^\vee}=\bq^{n}T_{(s_1s_0)^n}^{-1}.
\]
\end{enumerate}
In particular, under the geometric choice of dominance, we have $\rho=-1$. 
\end{ex}
\subsection{The asymptotic Hecke algebra}
\begin{dfn}
\emph{Lusztig's a-function} $a\colon\tilde{W}\to\N$ is defined such that $a(w)$ is the minimal value such that $q^{\frac{a(w)}{2}}h_{x,y,w}\in\mathcal{A}^+$ for all $x,y\in\tilde{W}$.
\end{dfn}
It is known that $a$ is constant on two-sided cells of $\tilde{W}$ and that 
\[
a(\cc)=\dim\Bb_u
\]
where $u$ is the unipotent conjugacy class in $G$ corresponding to $\cc$ under Lusztig's bijection.
It is also known that $a(w)\leq\ell(w)$ for all 
$w\in\tilde{W}$. 

In \cite{affineII} Lusztig defined an associative algebra
$J$ over $\Z$ equipped with an injection $\phi\colon H\into J\otimes_\Z\mathcal{A}$
which becomes an isomorphism after taking a certain completion
of both sides. As an abelian group, $J$ has a basis $\{t_w\}
_{w\in W}$. Recalling the Kazhdan-Lusztig basis elements
\[
C_w=\sum_{y\leq w}(-1)^{\ell(w)-\ell(y)}\bq^{\frac{\ell(w)}{2}-\ell(y)}P_{y,w}(\bq^{-1})T_y,
\]
the structure constants of $J$ are obtained from those in $H$ written in the $\{C_w\}_{w\in W}$-basis under the 
following procedure. Using the structure constants
\[
C_xC_y=\sum_{z\in W}h_{x,y,z}C_z
\]
for $h_{x,y,z}\in\mathcal{A}$, Lusztig then defines the integer $\gamma_{x,y,z}$ by the condition
\[
q^{\frac{a(z)}{2}}h_{x,y,z^{-1}}-\gamma_{x,y,z}\in q\mathcal{A}^+.
\]
The product in $J$ is then defined as
\[
t_xt_y=\sum_z\gamma_{x,y,z}t_{z^{-1}}.
\]
One then defines
\[
\phi(C_w)=\sum_{\substack{z\in W,~d\in\mathcal{D} \\ a(z)=a(d)}}h_{x,d,z}t_z,
\]
where $\mathcal{D}\subset\Waff$ is the set of distinguished involutions.
The elements $t_d$ for distinguished involutions $d$ are orthogonal idempotents. Moreover, 
$J=\bigoplus_{\mathbf{c}}J_{\mathbf{c}}$ is a direct sum of two-sided ideals indexed by
two-sided cells $\mathbf{c}\subset\Waff$. The unit element in each summand is 
$\sum_{\mathcal{D}\cap\mathbf{c}}t_d$, and the unit element of $J$ is $\sum_{d\in\mathcal{D}}t_d$.
\subsubsection{The lowest two-sided cell}
\label{subsection the lowest two-sided cell}
Let $\cc_0$ be the lowest two-sided cell, also called the ``big cell." It can be characterized
as the two-sided cell containing the longest element of $W$, and we have $a(\cc_0)=\ell(w_0)$.
The summand $J_0:=J_{\mathbf{c}_0}$ is particularly well-understood,
and has historically been the first summand for which any structure-theoretic result has been achieved
(consider, for example, the progression \cite{XiI}, \cite{XiAMemoir}, \cite{BO}).
Write $\phi_0$ for the composition of $\phi$ composed by the projection $J\otimes_\Z\mathcal{A}\to J_0\otimes_{\Z}\mathcal{A}$.

By \cite{XiI} and \cite{Nie}, we have the following description of $\cc_0\subset\Waff$.
Let $\tilde{\cc_0}$ be the lowest cell of the affine Weyl group of the universal covering 
group $\tilde{G}$ of $G$ with maximal torus $\tilde{T}$. Then $\cc_0=\tilde{\cc}_0\cap\Waff$, and 
\[
\tilde{\cc_0}=\sets{f^{-1}w_0\chi g}{f,g\in\Sigma,\chi\in X^*(\tilde{T})^+},
\]
where $\Sigma=\sets{wx_w}{x\in W}\subset\Waff(\tilde{G})$, where 
\begin{equation}
\label{eq Steinberg basis element definition}
x_w=w^{-1}\left(\prod_{\substack{\alpha\in\Delta \\ w^{-1}(\alpha)<0}}\varpi_\alpha \right)\in X^*(\tilde{T}),
\end{equation}
where $\varpi_\alpha$ is the fundamental dominant weight corresponding to $\alpha$. We have
\[
t_{f^{-1}w_0\lambda g}t_{(f')^{-1}w_0\nu g'}=0
\]
if $g\neq f'$, and 
\[
t_{f^{-1}w_0\lambda g}t_{g^{-1}w_0\nu g'}=\sum_{\mu}m_{\lambda,\nu}^\mu t_{f^{-1}w_0\mu g'},
\]
where $m_{\lambda,\nu}^\mu$ is the multiplicity of $V(\mu)$ in $V(\lambda)\otimes V(\nu)$.
\subsubsection{On a theorem of Steinberg}
Write $\pt=\Spec\C$.

Steinberg showed in \cite{St} that $K_{\tilde{T}}(\pt)$ is a free $K_{\tilde{G}}(\pt)$ module with basis
$\{x_w\}_{w\in W}$. 
Under the isomorphism $K_0(\pt/T)\simeq K_0(\Bb^\heart/\tilde{G})$, the $x_w$ define an 
$K_{0}(\pt/\tilde{G})$-basis $\{\Ff_w\}_w$ of the latter ring, where $\Ff_w=\Oo_{\Bb^\heart}(x_w)$, and \cite{KLDeligneLanglands} show that the natural pairing
\[
\pair{-}{-}\colon K_{0}(\Bb^\heart/\tilde{G})\otimes_{K_{0}(\pt/\tilde{G})}K_{0}(\Bb^\heart/\tilde{G})\to K_{0}(\pt/\tilde{G})
\]
is nondegenerate. While the dual basis is often employed in the literature starting from 
\textit{loc. cit.}, we are not aware of an explicit description of it. We provide one here in very low rank cases
in type $A$. The lack of a description in other cases of the dual basis elements as classes in $K$-theory of 
some natural objects of $\Coh(\Bb^\heart/\tilde{G})$ is the only obstruction to proving Proposition 
\ref{prop multiplication of sheaves tee_w} in greater generality.
\begin{lem}
\label{thm dual basis in K-theory}
Let $\tilde{G}=\SL_2$ or $\SL_3$.
The collection $\Gg_w=\Oo(y_w)[\ell(w)]$, where
\[
y_w= \left(w^{-1}\prod_{\substack{\alpha\in\Delta \\ w^{-1}(\alpha)>0}}\varpi_\alpha\right)\rho^{-1}
\]
defines the basis dual to Steinberg's basis of $K_{0}(\Bb^\heart/\tilde{G})$ under the above pairing. 
\end{lem}
\begin{ex}
In type $A_1$ and additive notation, we have $x_1=0$ and $x_{s_\alpha}=s_\alpha(\varpi_\alpha)=1-2=-1$. In this case the Steinberg
basis is self-dual, with $y_1=\varpi_\alpha-\rho=1-1=0$ and $y_{s_\alpha}=s_\alpha(0)-\rho=-1$.
\end{ex}
The lemma can be proved by direct computation.
\begin{rem}
\label{rem pairing failure}
The classes $[\Gg_w]$ cease to pair correctly with the Steinberg basis classes starting for $G=\SL_4$, in 
a way apparently governed by singularities of Schubert cells. For example, for $\SL_4$, one has
\[
\pair{[\Ff_w]}{[\Gg_1]}=\triv_{\SL_4}
\]
where $w=1$, or when $w=\sigma$ is the product of the two permutations in $\Sn_4$ that index singular Schubert varieties. In this case, the element dual to $[\Ff_1]$ is $[\Gg_1]+[\Gg_\sigma]$. We hope to produce natural
complexes in a future version of this paper that will lift these sums and pair correctly.
\end{rem}

\section{Sheaves and categories}
\label{section sheaves and categories}
All categories, functors, and schemes are derived unless indicated otherwise. We emphasize 
especially that all fibre products are derived (although frequently this consideration will have no effect).
Sections \ref{subsection derived schemes} and \ref{subsection Cat of Haff, Bez equiv} recall the necessary
material to define the category $\J$, and contain no new material.
\subsection{Derived stacks}
\label{subsection derived schemes}
Unless otherwise indicated, to ``apply base-change" means to apply Proposition 2.2.2 (b) 
of \cite{GRVol1}.

If $X$ is classical, then $\Coh(X)$ and $\Coh(X/G)$ are the usual bounded derived categories.
We write $\Rep(G):=\Coh(\ptG)$. We will often use silently the fact that if $f\colon X\to Y$ 
is a morphism of smooth locally-Noetherian schemes, then the pullback functor $f^*$ preserves coherence.
The classical schemes we work with will of course be exclusively locally-Noetherian, and the flag variety
and bundles over it are smooth.
\subsection{The scheme of singularities and singular support}
Given a coherent sheaf $\Ff$ on a stack $X$, Arinkin and Gaitsgory in \cite{AG} define
a classical stack $\sS(\Ff)$, the singular support of $\Ff$. The singular support serves in particular
to measure the extent to which an object of $\Coh(X)$ fails to lie in $\mathrm{Perf}(X)$. 
We will require only very special cases of the theory of singular support.

Let $X$ first be a quasi-smooth derived scheme. The classical scheme $\Sing(X)$ measures how far from being 
smooth $X$ is. Let $T^*(X)$ be the cotangent complex of $X$ and $T(X)$ its dual. Then one defines
\[
\Sing(X):=\relSpec\left(\Sym_{\Oo_{X^\heart}}H^1(T(X))\right)\to X^\heart. 
\]
The scheme of singularities is affine over $X^\heart$, but is not in general a vector bundle.
The singular support will be a conical subset of $\Sing(X)$. In general, if a morphism
\[
f\colon X\to Y
\]
exhibits $f^{-1}(\pt)$ as quasi-smooth, then given $x\in X$, 
\[
\Sing(X)_x=\coker(df_x)^*.
\]
Note that if $f\colon V\to W$ is a linear map between vector spaces, then the dg-algebra of functions on
the derived scheme $f^{-1}({0})$ is 
\[
\Oo_{\ker f}\otimes\Sym\left(\coker(f)^*[1]\right).
\]
The case of quotient stacks is analogous; one must only keep track of equivariance. See \cite[Sections 8,9]{AG}.
%
%
%
%
%
As derived schemes or stacks will appear below with approximately the same frequency as their truncations, there are no  notational savings to be had by adopting either the convention that all schemes are derived unless otherwise indicated, or the opposite convention. To match our convention about functors, we declare that in any case where a derived stack and its classical truncation appears, the derived stack will be without decoration, as will all classical stacks that appear without any derived enhancement.
\subsection{Categorification of $\Haff$, Bezrukavnikov's equivalence}
\label{subsection Cat of Haff, Bez equiv}
Let $\Nn=T^*(\Bb^\heart)$, and denote the Steinberg variety by 
$\St=\Nn\times^L_{\g}\Nn$. It is naturally a global complete intersection derived scheme, fitting into the pullback diagram
\begin{center}
\begin{tikzcd}
\St\arrow[d]\arrow[r]&\Nn\times\Nn\arrow[d, "f\circ i"]\\
\pt\arrow[r]&\g
\end{tikzcd}
\end{center}
where the vertical morphism is induced by the composite
\[
\Nn\times\Nn\to\g\oplus\g\to\g
\]
sending $(x,\bB, y,\bB')\mapsto x-y$. Therefore $\Sing(\St)$ is defined. Its fibre over a point 
$(X,\bB_1,\bB_2)$ of $\St^\heart$ is computed in \cite[Lemma 3.3.5]{ChenDhillon} to be
\begin{equation}
\label{eqn ChenDhillon Sing(St) fibres}
\Sing(\St)_{(X,\bB_1,\bB_2)}=\sets{Y\in\g}{Y\in \bB_1\cap\bB_2,~\kappa(Y,[X,-])=0}\subseteq\left(\g/(\nN_1\oplus\nN_2)\right)^*=\bB_1\cap\bB_2,
\end{equation}
where $\kappa$ is the Killing form.

The category $\Coh(\St/G\times\Gm)$ is monoidal under convolution of sheaves.
It categorifies $\HH$, and its unmixed $\Coh(\St/G)$ version is one side of Bezrukavnikov's celebrated equivalence \cite{tworel} upgrading the $K$-theoretic results we recall in the sequel. 
\begin{theorem}[Bezrukavnikov, \cite{tworel}]
Let $G^\vee$ be split over $F=\overline{\F_q}$ and dual to $G$. Let $I^\vee\subset\mathcal{L}G^\vee$ be an Iwahori subgroup of the loop group $\mathcal{L}G^\vee$. Then there is an equivalence of categories
\[
D_{I^\vee I^\vee}:=D^b(I^\vee\rquotient LG^\vee/I^\vee)\to \Coh(\St/G).
\]
The equivalence intertwines the automorphism of the left hand side induced by the Frobenius automorphism with 
pullback by the automorphism of $\St/G$ induced by $(X,\bB_1,\bB_2)\mapsto (qX, \bB_1,\bB_2)$.
\end{theorem}

\subsection{Categorification of $J_0$}
\subsubsection{Derived enhancement of the flag variety}
\label{subsubsection derived enhancement of the flag variety}
Let $|\E|$ be the total space of the quotient
\[
0\to\Nn\to \Bb^\heart\times\g\to|\E|\to 0
\]
and define
\begin{equation}
\label{eqn B defined by short exact sequence of vector bundles}
\Bb=\relSpec(\Sym_{\Oo_{\Bb^\heart}}\E[1])=\relSpec\left(\Sym_{\Oo_{\Bb^\heart}}\left(\Bb^\heart\times\g/\Nn\right)^*[1]\right).
\end{equation}
Then $\Bb$ is naturally a derived scheme with classical truncation $\Bb^\heart$ equipped with a morphism $i\colon\Bb\to\Nn$, and hence with a morphism
\[
i_{\mathrm{der}}\colon\Bb\times\Bb\to\St.
\]
By construction, we have a pullback diagram
\begin{equation}
\label{diagram B defined by fibre product}
\begin{tikzcd}
\Bb\arrow[d]\arrow[r]&\Nn\arrow[d]\\
\{0\}\arrow[r]&\Bb^\heart\times\g,
\end{tikzcd}
\end{equation}
where $\Bb^\heart\times\g\to\Bb^\heart$ is the trivial bundle with fibre $\g$ and $\{0\}$ is its zero-section. 
Therefore $\Bb$ is a quasi-smooth DG-scheme in the sense of \cite{AG}.
The description of $\Bb$ as a fibre product yields a similar description of $\Bb\times\Bb$. Indeed, the diagram
\begin{equation}
\label{diagram B B defined is fibre product}
\begin{tikzcd}
\Nn\arrow[r]\arrow[d, hook]&\g\arrow[d, "\id"]&\Nn\arrow[l]\arrow[d, hook]\\
\Bb^\heart\times\g\arrow[r]&\g&\Bb^\heart\times\g\arrow[l]\\
\{0\}\arrow[r]\arrow[u]&\pt\arrow[u]&\{0\}\arrow[l]\arrow[u]
\end{tikzcd}
\end{equation}
gives immediately the description
\begin{equation}
\label{diagram fibre product description of B x B}
\begin{tikzcd}
\Bb\times\Bb\arrow[r, "i_{\mathrm{der}}"]\arrow[d]&\St\arrow[d, "p_{\St}"]\\
\{0\}\arrow[r, "i_{\{0\}}"]&\Bb^\heart\times\Bb^\heart\times\g.
\end{tikzcd}
\end{equation}
where $\{0\}$ now means the zero-section of the trivial bundle $\Bb^\heart\times\Bb^\heart\times\g$.
We get the same description for the quotients by $G$ or $\GGm$. Note that \eqref{diagram fibre product description of B x B} says that $i_{\mathrm{der}}$ is a quasi-smooth closed immersion by 
\cite[Prop. 2.1.10]{AG}.

The category $\Coh(\BbG\times_{\ptG}\BbG)$ is a module category over the monoidal category $\Coh(\ptG)$,
via
\[
V\cdot\Ff=\pi^*V\otimes_{\Oo_{\Bb\times\Bb}}\Ff,
\]
for $V\in\Coh(\ptG)$, where $\pi\colon\BbG\times_{\ptG}\BbG\to\ptG$. The same procedure makes $\J$ into a
module category over $\Coh(\ptG)$. 
\begin{lem}
\label{lem Sing(ider) lands in diagonal}
If $\Ff\in\Coh_{G\times\Gm}(\St)$, then
\[
\sS(i^*_{\mathrm{der}}\Ff)\subseteq\Delta\tilde{\g}\subset\tilde{\g}\times\tilde{\g}.
\]
\end{lem}
\begin{proof}
We seek to apply Proposition 7.1.3 of \cite{AG} (c.f. Section 7.4.5 and Lemma 8.4.2 of \textit{op. cit.})  to 
the quasi-smooth closed embedding $i_{\mathrm{der}}$. We have
\begin{center}
\begin{tikzcd}
\Spec\Sym_{\Oo_{\Bb^\heart\times\Bb^\heart}}\left(i_{\mathrm{der}}^*T(\St)[1]\right)\arrow[rr,"\Sing(i_{\mathrm{der}})"]\arrow[dr]&&\Sing(\Bb\times\Bb)\arrow[dl]\\
&\St^\heart.
\end{tikzcd}
\end{center}
Fibrewise, the pairing condition of \eqref{eqn ChenDhillon Sing(St) fibres} becomes vacuous, and the singular codifferential is the linear map 
\begin{equation}
\label{eq Sing(ider) lands in diagonal}
\Sing(\St)_{(0,\bB_1,\bB_2)}=\left(\g/(\nN_1\oplus\nN_2)\right)^*=\bB_1\cap\bB_2\to(\g/\nN_1)^*\oplus(\g/\nN_2)^*=\bB_1\oplus\bB_2=\Sing(\Bb\times\Bb)_{(\bB_1,\bB_2)}
\end{equation}
induced by the direct sum of projections
\[
\g/\nN_1\oplus\g/\nN_2\onto\g/(\nN_1\oplus\nN_2).
\]
This implies that \eqref{eq Sing(ider) lands in diagonal} is the simply the diagonal embedding, and the lemma
follows.
\end{proof}

%
 

%
\subsubsection{Definition of the category $\J$}
We now define the category $\J$, the main point being the condition
we impose on the singular supports of its objects. 
In our case 
\[
\Sing(\Bb)=\relSpec\left(\Sym_{\Oo_{\Bb^\heart}}\E[2]\right)\to\Bb^\heart.
\]
Koszul duality gives an equivalence
\[
\mathrm{KD}\colon\Coh(\Bb)\to\Sym_{\Oo_{\Bb^\heart}}\mathcal{E}[2]-\Mod^{\mathrm{f.g.}},
\]
and following \cite{AG}, we set 
\[
\sS(\Ff)=\supp(\mathrm{KD}(\Ff))\subset |\E|.
\]
Thus for $\Ff=\Ff_1\boxtimes\Ff_2\in\Coh(\Bb\times\Bb)$, $\sS(\Ff)$ is just the usual support of some other 
sheaf on the total space of the bundle $\Sing(\Bb)\times\Sing(\Bb)$. We note that $\Sing(\Bb)$ is 
none other than the bundle $\tilde{\g}$. Indeed, the fibres of 
$\Bb$ are 
\[
\Spec\Sym_\C((\g/\nN)^*[1])=\Spec\Sym_\C(\bB[1])
\]
and Koszul duality identifies
\[
\Sym_\C(\bB[1])-\Mod\simeq\Sym_\C(\bB^*[2])-\Mod,
\]
and it makes sense to take the support of a module on the right-hand side on the scheme $\bB$, by defining the support to be the 
support of the cohomology over the classical ring $\Sym_\C\bB^*$.

We define $\J$ to be the full subcategory of $\Coh_{G}(\Bb\times\Bb)$ with objects $\Ff$
such that the projection
\[
\sS(\Ff)\to\Sing(\Bb)
\]
onto the first factor is a proper morphism. We write
\[
\J:=\Coh(\BbG\times_{\ptG}\BbG)_!
\]
and
\[
\JA:=\Coh(\BbGGm\times_{\ptGGm}\BbGGm)_!,
\]
where $\Gm$ acts trivially on $\Bb$.

There are two obvious ways that the projection onto the first factor
can be proper: either $\mathrm{KD}(\Ff)$ is of form $\Delta_*\Ff'$, where $\Delta$
is the diagonal, or $\mathrm{KD}(\Ff)=\Ff_1'\boxtimes\Ff_2'$ with $\supp(\Ff_2)$ contained in the zero section.
This latter case arises precisely from sheaves $\Ff_1\boxtimes\Ff_2\in\J$ such that $\Ff_2$ is perfect.
These are essentially the only examples that we will encounter. The image of $i_{\mathrm{der}}^*$
consists of sheaves of the first type (this is especially easy to see for those sheaves whose images in 
$K$-theory are contained in $Z(J_0)=\phi_0(Z(\Haff))$), and the sheaves $\tee_w$ 
that we define in Section \ref{subsubsection the sheaves tee_w} are all examples of the second kind.
\begin{rem}
\label{rem singsupp reflects phi0 not surjective}
This fact, together with the second statement of the main theorem, can be viewed as a categorification 
of the fact that $\phi_0$ is injective but not surjective.
\end{rem}
Consider the pairing operation defined by
\[
\pair{\Ff}{\Gg}= \pi_*(\Ff\otimes\Gg)
\]
where $\pi\colon\BbG\to\ptG$. In general, this operation does not define a functor
\[
\pair{-}{-}\colon\Coh(\BbG)\otimes\Coh(\BbG)\to\ptG,
\]
but will do so when it comes to convolution of objects of $\J$.
\begin{prop}
\label{proposition J0 monoidal and ider pullback is monoidal}
The category $\J$ is a monoidal category under convolution of sheaves, 
and the pullback $i_{\mathrm{der}}^*$ defines a monoidal functor
\[
i_{\mathrm{der}}^*\colon\Coh(\St/G\times\Gm)\to\JA
\]
such that 
\[
\sS(i_{\mathrm{der}}^*\Ff)\subset\Delta\tilde{\g}
\]
for all $\Ff$. The category $\J$ is a triangulated subcategory of $\Coh_G(\Bb\times\Bb)$.

Additionally,
\begin{enumerate}
\item 
If $\Ff_1\boxtimes\Gg$ and $\Gg'\boxtimes\Ff_2$ are in $\J$, then $\pair{\Gg}{\Gg'}\in\Coh(\ptG)$ and
\[
\Ff_1\boxtimes\Gg\star\Gg'\boxtimes\Ff_2=\pair{\Gg}{\Gg'}\Ff_1\boxtimes\Ff_2;
\]
\item
If $V_1,V_2\in\Coh(\ptG)$, then
\[
(V_1\cdot\Ff)\star(V_2\cdot\Gg)=(V_1\otimes_\C V_2)\cdot\Ff\star\Gg
\]
for $\Ff,\Gg\in\J$.
\end{enumerate}
\end{prop}
\begin{rem}
The category $\Coh(\BbG\times_{\ptG}\BbG)$ is not monoidal.
\end{rem}
\begin{proof}
Let 
\[
\Ff_1\to\Ff_2\to\Ff_3\to
\]
be a distinguished triangle in $\Coh(\BbG\times_{\ptG}\BbG)$ such that
$\Ff_1,\Ff_2\in\J$. Applying Koszul duality, we get
\[
\KD(\Ff_1)\to\KD(\Ff_2)\to\KD(\Ff_3)\to
\]
in $\Coh(\tilde{\g}/G\times_{\ptG}\tilde{\g}/G)$. Then
\[
\supp\KD(\Ff_3)\subset\supp\KD(\Ff_1)\cup\supp\KD(\Ff_2),
\]
and we see that projection $\supp\KD(\Ff_3)\to\tilde{\g}/G$ onto the first factor is proper.

Let $\Ff\in\J$ and let $\Gg\in\Coh(\BbG\times_{\ptG}\BbG)$. Then their convolution will be coherent if
\[
\left(\Ff_{12}\boxtimes\Oo_{\BbG}\right)\otimes\left(\Oo_{\BbG}\boxtimes\Gg_{23}\right),
\]
is coherent, where the subscripts $ij$ indicate which factors inside $\BbG\times_{\ptG}\BbG\times_{\ptG}\BbG$ a 
given sheaf sits on. Noting that $\sS(\Oo_\BbG)=\{0\}$, we have
\begin{equation}
\label{eq prop 1 SS product}
\sS(\Ff_{12}\boxtimes\Oo_{\BbG})\cap \sS(\Oo_{\BbG}\boxtimes\Gg_{23})\subset \{0\}/G\times_{\ptG} \tilde{\g}/G\times_{\ptG} \{0\}/G.
\end{equation}
We claim that this intersection is in fact contained in the zero section of 
$\tilde{\g}/G\times_{\ptG}\tilde{\g}/G\times_{\ptG}\tilde{\g}/G$. First,
projection from $\sS(\Ff_{12})\times\{0\}$ to the first coordinate is a proper morphism, and so the same 
is true for projection from the intersection; $\{0\}\times\sS(\Gg_{23})$ is closed. As $\sS(\Ff_{12})$
is a conical subset, it now follows that the intersection is contained in $\{0\}\times\{0\}\times\{0\}$.
The claim now follows from \cite{AG} Corollary 8.4.8. 

Now suppose that projection to the first factor from $\sS(\Gg_{23})$ is also proper. We have
\[
\Sing(p_{13})\colon\tilde{\g}/G\times_{\ptG}\tilde{\g}/G\times\Bb^\heart/G\to\tilde{\g}/G\times_{\ptG}\tilde{\g}/G\times_{\ptG}\tilde{\g}/G
\]
is the inclusion of the zero section into the second coordinate, and is the identity on the other coordinates.

Define
\[
Y_2=\sets{(x_1,x_3)}{(x_1,z)\in\sS(\Ff_{12}),~(z,x_3)\in\sS(\Gg_{23})~\text{for some}~z\in\{0\}}.
\]
Then 
\begin{align*}
&\Sing(p_{13})^{-1}\left(\sets{(x_1,x_2,x_3)}{(x_1,x_2)\in\sS(\Ff_{12}),~(x_2,x_3)\in\sS(\Gg_{23})}\right)
\\
&=
\sets{(x_1,z,x_3)}{(x_1,z)\in\sS(\Ff_{12}),~(x_2,z)\in\sS(\Gg_{23}),~z\in\{0\}}\\
&=
Y_2\times_{\BbG\times_{\ptG}\BbG}\BbG\times_{\ptG}\BbG\times_{\ptG}\BbG.
\end{align*}
It follows from \cite{AG}, Lemma 8.4.5 that $\sS(\Ff\star\Gg)\subset Y_2$. It therefore
suffices to show that the projection $Y_2\to\tilde{\g}/G$ is proper. Indeed, we have 
\[
p_1^{-1}(K)\subset K\times p_{2,\Gg}\left(p_{1,\Gg}^{-1}(\{0\})\right),
\]
for any $K$, where $p_{i,\Gg}$ is the projection $\sS(\Gg)\to\tilde{\g}/G$ onto the $i$-th factor.

Therefore $\J$ is a monoidal category. The same of course goes for $\JA$.

Now we show the first formula. It is easy to see that if $\Ff\boxtimes\Gg\in\J$, then $\sS(\Gg)$ must be contained
in the zero section, \textit{i.e.} that $\Gg$ must be perfect.
Then, $\pair{\Gg}{\Gg'}$ is coherent because $\Gg\otimes\Gg'$ is and the map to 
$\pt$ is proper. Its pullback to $\Bb\times\Bb$ is then perfect. The remainder of the formula is obtained
by carrying out the calculations in Lemma 5.2.28 of \cite{CG}. The required projection formula and 
base-change are provided by Lemma 3.2.4 and Proposition 2.2.2 (b) of \cite{GRVol1}, respectively.

We next claim that if $\Ff\in\DGCoh_{G\times\Gm}(\St)$, then $i^*_{\mathrm{der}}\Ff$ is coherent, and the 
projection
\[
p\colon\sS(i^*_{\mathrm{der}}\Ff)\subset\tilde{\g}/G\times_{\ptG}\tilde{\g}/G\to\tilde{\g}/G
\]
onto the first factor is proper. Coherence follows again from \eqref{diagram B B defined is fibre product}.
Indeed, we need only show that the pushforward of $i^*_{\mathrm{der}}\Ff$ to $\Bb^\heart\times\Bb^\heart$
is coherent, and base-change says that this equals $i_{\{0\}}^*p_{\St*}\Ff$. By hypothesis
$p_{\St*}\Ff$ is coherent, and hence by smoothness of $\Bb^\heart\times\Bb^\heart\times\g$ the pullback
is also coherent. We must check that $i_{\mathrm{der}}^*\Ff\in\J$. Indeed, this follows immediately from Lemma \ref{lem Sing(ider) lands in diagonal}, which says that $\sS(i_{\mathrm{der}}^*\Ff)\subset\Delta \tilde{\g}$.

We now check that $i_{\mathrm{der}}^*$ is monoidal. Diagrams \eqref{diagram B defined by fibre product} and
\begin{center}
\begin{tikzcd}
\Nn/\GGm\arrow[r]\arrow[d]&\g/\GGm\arrow[d]&\g/\GGm\arrow[l]\arrow[d]\\
\BbGGm^\heart\arrow[r]&\ptGGm&\g/\GGm\arrow[l]\\
\BbGGm^\heart\arrow[u]\arrow[r]&\ptGGm\arrow[u] &\ptGGm\arrow[u]\arrow[l]
\end{tikzcd}
\end{center}
imply that
\[
\Nn/\GGm\times_{\BbGGm^\heart\times_{\ptGGm}\g/\GGm}\BbGGm^\heart\simeq\Nn/\GGm\times_{\g/\GGm}\ptGGm\simeq\BbGGm.
\]
Thus

\begin{align*}
&\BbGGm\times_{\ptGGm}\BbGGm\times_{\St/\GGm}\Nn/\GGm\times_{\g/\GGm}\Nn/\GGm\times_{\g/\GGm}\Nn/\GGm
\\
&\simeq\BbGGm\times_{\ptGGm}\BbGGm\times_{\g/\GGm}\Nn/\GGm
\\
&\simeq\BbGGm\times_{\ptGGm}\BbGGm\times_{\ptGGm}\ptGGm\times_{\g/\GGm}\Nn
\\
&\simeq\BbGGm\times_{\ptGGm}\BbGGm\times_{\ptGGm}\BbGGm,
\end{align*}

and we can apply base-change to the pullback diagram
\begin{equation}
\label{diagram ider monoidal base change}
\begin{tikzcd}
\BbGGm\times_{\ptGGm}\BbGGm\times_{\ptGGm}\BbGGm
\arrow[r, "i\times i\times i"]\arrow[d, "\pi_{ij}"]
&\Nn/\GGm\times^L_{\g/\GGm}\Nn/\GGm\times^L_{\g/\GGm}\Nn/\GGm
\arrow[d, "p_{ij}"]\\
\BbGGm\times_{\ptGGm}\BbGGm\arrow[r, "i_{\mathrm{der}}"]&\St/\GGm
\end{tikzcd} 
\end{equation}
where $\pi_{ij}$ and $p_{ij}$ are the projections.
%
%

With diagram \eqref{diagram ider monoidal base change} in hand, the remainder is entirely formal. Indeed, according to the 
definition of convolution on $\St$, we compute as follows: Let $\Ff,\Gg\in\Coh(\St/\GGm)$.
Then 
\begin{align}
i_{\mathrm{der}}^*(\Ff\star\Gg)
&=
i_{\mathrm{der}}^*p_{13*}\left(p_{12}^*\Ff\otimes p_{23}^*\Gg\right)
\label{eq ider monoidal before bc}
\\
&\simeq
\pi_{13*}(i\times i\times i)^*\left(p_{12}^*\Ff\otimes p_{23}^*\Gg\right)
\label{eq ider monoidal after bc}
\\
&\simeq
\pi_{13*}\left(\pi_{12}^*i_{\mathrm{der}}^*\Ff\otimes\pi_{23}^*i_{\mathrm{der}}^*\Gg\right)
\label{eq ider monoidal factorization}
\\
&=
i_{\mathrm{der}}^*\Ff\star i_{\mathrm{der}}^*\Gg.
\nonumber
\end{align}
We used base-change 
for the diagram \eqref{diagram ider monoidal base change} with $ij=13$ between lines 
\eqref{eq ider monoidal before bc} and \eqref{eq ider monoidal after bc}, and just commutativity 
of \eqref{diagram ider monoidal base change} for $ij=12$ and $ij=23$ on line \eqref{eq ider monoidal factorization}.
\end{proof}
Recall that $a(\cc_0)=\dim\Bb$, any quasicoherent sheaf on $\Bb^\heart$ has cohomology only in degrees 
at most $a(\cc_0)$. On the level of $K$-theory, this reflects the influence of 
the $a$-function on the multiplication in $J$. 
%
%
\subsubsection{The sheaves $\tee_w$}
\label{subsubsection the sheaves tee_w}
Xi, in \cite{XiI} for $G$ simply-connected, and Nie in \cite{Nie} in general gave a description in $K$-theory
of the elements $t_w$ for $w\in\cc_0$. We recall this construction below in Section \ref{subsection the lowest two-sided cell}; here we follow it on the level of categories in the special case $G=\SL_2$ or $\SL_3$, where it can be carried out almost verbatim. As remarked above, we hope to move beyond these two special cases in 
future work.

Recalling the equivalence
\[
\Ind_B^G\colon\Coh(\pt/B)\to\Coh(\BbG^\heart),
\]
we define
\[
\Ff_w=\Ind_B^G\mathrm{Infl}_T^B(x_w),
\]
where $x_w\in\Coh_T(\pt)$ is as in \eqref{eq Steinberg basis element definition} and 
\[
\mathrm{Infl}_T^B\colon\Coh(\pt/T)\to\Coh(\pt/B)
\]
is inflation. Likewise, define
\[
\Gg_w=\Ind_B^G\mathrm{Infl}_T^B(y_w)
\] 
where $y_w$ is the dual basis from Lemma \ref{thm dual basis in K-theory}.

Now if $w=fw_0 g^{-1}$, define
\[
\tee_w=\Ff_f\boxtimes p^*\Gg_g,
\]
where $p\colon\Bb\to\Bb^\heart$,
and if $w=fw_0 \chi g^{-1}$, define
\[
\tee_w=V(\chi)\tee_{fw_0g^{-1}},
\]
where we view $\Ff_f$ as pushed forward under the inclusion of the zero section of $\Bb$. Clearly
$\tee_w\in\Coh_G(\Bb\times\Bb)$. Moreover, as $\Bb^\heart$ is smooth, $\Gg_g$ is perfect, and hence
$p^*\Gg_g$ is perfect. Therefore $\sS(\Gg)$ is contained in the zero section of $\Sing(\Bb)$, 
and $\tee_w\in\J$. 
%
%

By Proposition \ref{proposition J0 monoidal and ider pullback is monoidal} (or using Corollary 8.4.8 of \cite{AG} directly), $\pair{\Gg_g}{\Ff_f}$ is defined for all $g,f$ and takes values in $\Coh(\pt/G)$. In fact, it agrees with the pairing $\pair{-}{-}_\heart$ on the classical truncation given the by the same procedure:
\[
\pair{p^*\Gg_g}{\Ff_{f}}=\pi_*(p^*\Gg_g\otimes\Ff_{f})=\pi_*^\heart p_*(p^*\Gg_g\otimes\Ff_{f})
=\pi_*^\heart\left(\Gg_g\otimes p_*\Ff_{f}\right)=\pair{\Gg_g}{\Ff_{f}}_\heart
\]
where
\begin{center}
\begin{tikzcd}
\BbG\arrow[rr, "p"]\arrow[dr, "\pi"]&&\BbG^\heart\arrow[dl, "\pi^\heart"]\\
&\ptG.
\end{tikzcd}
\end{center}
Therefore by the Borel-Weil-Bott theorem, we have
\[
\Ff_f\boxtimes p^*\Gg_g\star\Ff_{f'}\boxtimes p^*\Gg_{g'}=
\begin{cases}
\Ff_f\boxtimes\Gg_{g'}&\text{if}~g=f'\\
0&\text{otherwise}
\end{cases}.
\]
Moreover, as $G$ is reductive, we have again by Proposition \ref{proposition J0 monoidal and ider pullback is monoidal} that
\[
(V(\lambda)\Ff_f\boxtimes p^*\Gg_g)\star (V(\nu)\Ff_g\boxtimes p^*\Gg_{g'})
=\left(V(\lambda)\otimes V(\nu)\right)\Ff_f\boxtimes p^*\Gg_{g'}=\bigoplus_{\mu}V(\mu)(\Ff_f\boxtimes p^*\Gg_g)^{\oplus m^\mu_{\lambda,\nu}},
\]
where $m^\mu_{\lambda,\nu}$ is the multiplicity of $V(\mu)$ in $V(\lambda)\otimes V(\nu)$.

\section{$K$-theory}
In this section we show that the functor $\i_{\text{der}}^*$ categorifies Lusztig's homomorphism $\phi_0$.
By $K$-theory we shall always mean simply the Grothendieck group. We write $R(G)=K_{G}(\pt)=K_0(\Rep(G))$.
\subsection{$K$-theory of classical schemes and Lusztig's homomorphism}
We first relate Lusztig's homomorphism to a construction in $K$-theory of classical schemes given
in \cite{CG}. There is no new material in this section; when $G$ is simply-connected the relationship
is given by \cite{XiIII}, and the analogous result in general follows from \cite{BO}.
In order to perform calculations, though, we must devote significant space to fixing conventions.

Recall that the $K_{G\times\Gm}(\St)\simeq H$ as $\mathcal{A}$-algebras. We will use the explicit isomorphism 
given by Chriss and Ginzburg in \cite{CG}, Theorem 7.2.5.

Recall that by \cite{BO}, $J_0\simeq K_{G}(\mathbf{Y}\times\mathbf{Y})$ for a centrally-extended set
$\mathbf{Y}$ of cardinality $\# W$. Moreover, by 5.5 (a) of \textit{loc. cit.}, the stabilizer
of every $y\in\mathbf{Y}$ is $G$. When $G$ is simply-connected, it has no nontrivial central
extensions, and hence in this case $J_0\simeq\Mat_{\# W}(R(G))$ is a matrix ring,
as first shown in \cite{XiI}. Combining these results, we obtain an injection 
$\varphi_1\colon J_0\into\Mat_{\# W}(R(\tilde{G}))$, where $\tilde{G}\onto G$ is a simply-connected.
Write $H(\Waff(\tilde{G}))$ for the affine Hecke algebra of $\tilde{G}$.

In parallel, when $G$ is simply-connected, the external tensor product gives an isomorphism 
$K_0(\BbG\times_{\ptG}\BbG)\simeq K(\BbG)\otimes_{R(G)}K_0(\BbG)\simeq\Mat_{\# W}(R(G))$, by 
\cite{CG}, Theorem 6.2.4. This theorem does not hold when $G$ is not simply-connected. Indeed,
the one-dimensional tempered representations of $H$ that are not discrete series give 
of $J_0$-modules for $G^\vee=\SL_2$. On the other hand, the external tensor product still gives an inclusion
\[
\psi_1\colon K_0(\BbG\times_{\ptG}\BbG)\into K_0(\Bb/\tilde{G}\times_{\pt/\tilde{G}}\Bb/\tilde{G})\simeq\Mat_{\# W}(R(\tilde{G})).
\]
(Note that $G$ and $\tilde{G}$ have canonically isomorphic flag varieties and Weyl groups.)
%

By \cite{Nie}, we have an isomorphism $\sigma\colon J_0\to K_0(\BbG^\heart\times_{\ptG}\BbG^\heart)$ regardless 
of whether $G$ is simply-connected or not. 
\begin{lem}[\cite{XiIII}]
\label{lem CG map and Lusztig map conjugate}
The following diagram of $\mathcal{A}$-algebras
\begin{center}
\begin{tikzcd}
K_0(\St/\GGm)\arrow[rr]&&K_0(\Nn/\GGm\times_{\ptGGm}\Nn/\GGm)
\arrow[d, "\bar{\iota}^*\circ\bar{p}_*", hook]\\
&&K_{0}(\BbGGm\times_{\ptGGm}\BbGGm)\arrow[d,"\psi_1", hook]\\
&&\Mat_{\# W\times \#W}(R(\tilde{G}\times\Gm))\\
\HH\arrow[uuu, "\sim"]\arrow[r, "\phi_0"]&J_0\otimes_\Z\mathcal{A}\arrow[r, "\varphi_1", hook]&\Mat_{\# W\times \#W}(R(\tilde{G}\times\Gm))\arrow[u, "\Ad(A)"]
\end{tikzcd}
\end{center}
commutes, where $A$ is the change-of-basis matrix from the the $Z(H(\tilde{W}(\tilde{G}))$-basis 
$\sets{\theta_{e_w}C}{w\in W}$
of $H(\tilde{W}(\tilde{G}))$ to the $Z(H(\tilde{W}(\tilde{G})))$-basis $\sets{C_{d_ww_0}}{w\in W}$, where $e_w$ and $d_w$ are as in \cite{XiIII}.
\end{lem}
%
%

\subsubsection{$K$-theory of derived schemes and Lusztig's homomorphism}

Koszul duality identifies $\Coh(\BbGGm)$ with $\Coh(\tilde{\g}^*[-2]/\GGm)$.

We now establish the relationship in $K$-theory between the monoidal functor $i_{\mathrm{der}}^*$
from Section \ref{subsubsection derived enhancement of the flag variety} and the morphism $\phi_0$. Write
\[
K_{0}(\St/\GGm):=K_0\left(\Coh(\St/\GGm)\right)
\]
and write $K(X^\heart):=K_0(\Coh(X^\heart))$ whenever $X^\heart$ is a classical scheme, and similarly for 
stacks.

If $X$ is a derived stack with classical truncation $X^\heart$, we may define a morphism
\[
K(X)\to K(X^\heart)
\]
by 
\begin{equation}
\label{eq devissage isom}
[\Ff]\mapsto\sum_{i}(-1)^i[\pi_i(\Ff)],
\end{equation}
where $\pi_i(F)$ is viewed as a $\pi_0(\Oo_X)$-module. 

Recalling that the derived structure on $X$ is to be thought of as ``higher nilpotents," this morphism is 
identical in spirit to identifying $K_0(\Coh(\Spec A))$ and 
$K_0(\Coh(\Spec A_{\mathrm{red}}))$, where $A$ is Noetherian and $A_{\mathrm{red}}$ is reduced. 
Indeed, the map \eqref{eq devissage isom} is also an isomorphism of abelian groups, and both isomorphisms are 
consequences of d\'{e}vissage; see e.g. \cite{ToenEMS}.
\begin{lem}
Pushforward by bundle projection $p\colon \St/\GGm\to\St/\GGm^\heart$ induces the map \eqref{eq devissage isom} on $K$-theory. This map is an isomorphism of rings.
\end{lem}
\begin{proof}
By the remarks preceding the lemma, it suffices to show that $p_*$ respects convolution in $K$-theory. By definition, we have
\begin{equation}
\label{eq pushforward is ring hom 1}
p_*([\Ff])\star p_*([\Gg])=\sum_{i,j}(-1)^{i+j}\pi_i(\Ff)\star\pi_j(\Gg),
\end{equation}
whereas
\[
p_*([\Ff]\star[\Gg])=p_*p_{13*}(p_{12}^*\Ff\otimes_{q_2^*\Oo_{\St}}p_{23}^*\Gg)=
p_{13*}^\heart p_{3*}(p_{12}^*\Ff\otimes_{q_2^*\Oo_{\St}}p_{23}^*\Gg)
\]
by commutativity of the diagram
\begin{center}
\begin{tikzcd}
\left(\left(\Nn\times_\g\Nn\right)\times_{\Nn\times_\g\Nn}\left(\Nn\times_\g\Nn\times_\g\Nn\right)\right)/\GGm
\arrow[d, "p_3"]
\arrow[r, "p_{13}"]&\St/\GGm\arrow[d, "p_{\St}"]\\
\Nn/\GGm\times_{\g/\GGm}\times_{\g/\GGm}\Nn/\GGm\arrow[r, "p^\heart_{13}"]&\St/\GGm^\heart.
\end{tikzcd}
\end{center}
We have
\begin{equation}
\label{eq pushforward is ring hom 2}
p_{13*}^\heart p_{3*}(p_{12}^*\Ff\otimes_{q_2^*\Oo_{\St}}p_{23}^*\Gg)=
p_{13*}^\heart\sum_{n}(-1)^n\pi_n\left(p_{12}^*\Ff\otimes_{q_2^*\Oo_{\St}}p_{23}^*\Gg\right),
\end{equation}
and so for \eqref{eq pushforward is ring hom 1} to equal \eqref{eq pushforward is ring hom 2}, we need
\[
\pi_n\left(p_{12}^*\Ff\otimes_{q_2^*\Oo_{\St}}p_{23}^*\Gg\right)=\sum_{i+j=n}
p_{12}^{\heart *}\pi_i(\Ff)\otimes_{\Nn\times\Nn\times\Nn}p_{23}^{\heart *}\pi_j(\Gg),
\]
which follows from the K\"{u}nneth formula.
%
%
\end{proof}
In the case of $K_{0}(\J)$, we can define another map to the $K$-theory of the truncation.
Recall $p\colon\BbGGm\to\BbGGm^\heart$ is the bundle projection morphism, and let $i\colon\BbGGm^\heart\to\BbGGm$ be 
the inclusion of the zero section. Then define $\Phi$ to be the composite
\[
\Phi\colon K(\J)\overset{\id\times i^*}{\to}K_0(\BbGGm\times_{\ptGGm}\BbGGm^\heart)\overset{p_*\times\id}{\to}K_G(\BbGGm^\heart\times_{\ptGGm}\BbGGm^\heart).
\]
That is, if $\Ff\boxtimes\Gg\in\J$, then
\[
\Phi([\Ff\boxtimes\Gg])=[p_*\Ff]\boxtimes[i^*\Gg].
\]
\begin{rem}
It is necessary that the source of $(\id\times i)^*$ (which we will show makes sense as a functor) is $\J$ and 
not all of $\Coh_G(\Bb\times\Bb)$; the functor
\[
i^*\colon\QCoh(\BbGGm)\to\QCoh(\BbGGm^\heart)
\]
does not preserve coherence in general.
\end{rem}
\begin{lem}
\label{lem Phi is well-defined}
The morphism $\Phi$ is well-defined and is a surjective morphism of $R(G\times\Gm)$-algebras.
\end{lem}
\begin{proof}
By the K\"unneth formula, we have $\Sing(\Bb\times\Bb^\heart)\simeq\tilde{\g}/\GGm\times\Bb^\heart/\GGm$. Then
\[
\Sing(\id\times i)\colon\tilde{\g}/\GGm\times_{\ptGGm}\tilde{\g}/\GGm\to\tilde{\g}/\GGm\times_{\ptGGm}\Bb^\heart/\GGm
\]
is given by the identity in the first coordinate, and then the zero map in the second coordinate.

Therefore $\Sing(\id\times i)^{-1}(\{0\})=\{0\}\times\tilde{\g}/\GGm$. Now let $\Ff\in\J$. The argument is essentially the 
same as in the proof of Proposition \ref{proposition J0 monoidal and ider pullback is monoidal}.
The set-theoretic intersection
\begin{multline*}
\left(\sS(\Ff)\times_{(\Bb\times\Bb)/\GGm}\BbGGm\times_{\ptGGm}\BbGGm^\heart\right)\cap\Sing(\id\times i)^{-1}(\{0\})
\\
=\sS(\Ff)\cap(\{0\}\times\tilde{\g})/\GGm
\end{multline*}
is contained in the zero-section of $\Sing(\BbGGm\times_{\ptGGm}\BbGGm)$. Indeed, $\Ff\in\J$ and $\Bb^\heart$ is compact,
so the above intersection must be compact. As $\sS(\Ff)$ is conical, we see the intersection must be contained
in the zero-section. We conclude  by \cite{AG}, Corollary 8.4.8 that $(\id\times i)^*$ is well-defined.
Obviously $p_*\times\id$ is well-defined, and hence $\Phi$ is well-defined. As each of $(\id\times i)^*$
and $p_*\times\id$ are $R(G)$-linear (the latter by the projection formula), so is $\Phi$.

Finally, we show that $\Phi$ is a morphism of rings. Using linearity, we compute
\[
\Phi([\Ff_1]\boxtimes[\Gg])\star\Phi([\Gg']\boxtimes[\Ff_2])=p_*[\Ff_1]\boxtimes i_{\mathrm{der}}^*[\Gg]\star p_*[\Gg']\boxtimes i_{\mathrm{der}}^*[\Ff_2]=
\pair{i_{\mathrm{der}}^*[\Gg]}{p_*[\Gg']}_{\Bb^\heart}\cdot p_*[\Ff_1]\boxtimes i_{\mathrm{der}}^*[\Ff_2],
\]
whereas 
\[
\Phi([\Ff_1]\boxtimes[\Gg]\star[\Gg']\boxtimes[\Ff_2])=\pair{[\Gg]}{[\Gg']}_{\Bb}\cdot p_*[\Ff_1]\boxtimes i_{\mathrm{der}}^*[\Ff_2].
\]
Write $\pi\colon\Bb/\GGm^\heart\to\ptGGm$. Then
\[
\pair{i_{\mathrm{der}}^*[\Gg]}{p_*[\Gg']}_{\Bb^\heart}=\pi_*(i_{\mathrm{der}}^*[\Gg]\otimes_{\Oo_{\Bb^\heart}}p_*[\Gg'])
=\pi_*p_*\left(p^*i_{\mathrm{der}}^*[\Gg]\otimes_{\Oo_{\Bb}}[\Gg']\right)=\pair{[\Gg]}{[\Gg']}_{\Bb}
\]
by the formulation of the projection formula in \cite{GRVol1}, Lemma 3.2.4. 
Surjectivity follows as $\Phi$ has a section defined by
$[\Ff]\boxtimes[\Gg]\mapsto i_*[\Ff]\boxtimes p^*[\Gg]$. This completes the proof.
\end{proof}
%
%
\begin{lem}
\label{lem Phi preserves structure sheaf of diagonal}
We have $\Phi(\Oo_{\Delta\Bb}(\lambda))=\Oo_{\Delta\Bb^\heart}(\lambda)$.
\end{lem}
\begin{proof}
This is a local computation that amounts to the map 
\[
\C[x,\epsilon]\otimes\C[y,\delta]/(x-y,\epsilon-\delta)\to\C[x,y]
\]
quotienting by $\delta$ and leaving the first factor untouched, where $|x|=|y|=0$ and $|\epsilon|=|\delta|=-1$. One sees that quotienting by $\delta$ also kills $\epsilon$.
\end{proof}
%
%
\begin{prop}
\label{prop K-theory morphisms conjugate}
The following diagram of $R(\Gm)$-algebras 
\begin{center}
\begin{tikzcd}
K_{0}(\St/\GGm)\arrow[rr, "i_{\mathrm{der}}^*"]\arrow[d, "p_{\St*}"]&&K_{0}(\J)\arrow[d, "\Phi"]
\\
K_{0}(\St^\heart/\GGm)\arrow[r]&K_{0}(\Nn/\GGm\times_{\ptGGm}\Nn/\GGm)
\arrow[r, "\bar{\iota}^*\circ\bar{p}_*"]&K_{0}(\Bb^\heart\times\Bb^\heart)/\GGm)
\arrow[d,"\psi_1\otimes\id_{\mathcal{A}}", hook]
\\
&&\Mat_{\# W\times \#W}(R(\tilde{G}\times\Gm))
\\
&&\Mat_{\# W\times \#W}(R(\tilde{G}\times\Gm))\arrow[u, "\Ad(A)"]
\\
\HH\arrow[uuu, "\sim"]\arrow[rr, "\phi_0"]&&J_0\otimes_\Z\mathcal{A}\arrow[u, "\phi_1\otimes\id_{\mathcal{A}}", hook]
\end{tikzcd}
\end{center}
commutes. 
\end{prop}
We will first describe the middle morphism on the $K$-theory of the classical schemes. Consider the diagrams
\begin{equation}
\label{diagram Bernstein subalgebra images classical K-theory}
\begin{tikzcd}
\Nn/\GGm\arrow[d, "\iota_\Delta", hook]\arrow[r, "\pi_\Delta"]&\BbGGm^\heart_\Delta\arrow[dd, "\Delta"]\\
\St/\GGm^\heart\arrow[d, "\bar{p}", hook]\\
\BbGGm^\heart\times_{\ptGGm} \Nn/\GGm \arrow[r, "\id\times\pi"]&\BbGGm\times_{\ptGGm}\BbGGm,
\end{tikzcd}
\end{equation}
which is Cartesian, and the diagram
\begin{equation}
\label{diagram image of C's in classical K-theory}
\begin{tikzcd}
(\Pp^1\times\Pp^1\arrow[d, hook, "\bar{\iota}"])/\GGm&\\
(T^*\Pp^1\times\Pp^1)/\GGm &\St/\GGm\arrow[l, hook, "\bar{p}" above]=(T^*\Pp^1\times_{\tilde{N}^\vee}T^*\Pp^1)/\GGm\\
&(\Pp^1\times\Pp^1)/\GGm\arrow[u, hook, "j"]\arrow[ul, hook, "\bar{\iota}"]
\end{tikzcd}
\end{equation}
which relates to the case $G=\SL_2$. Put
\[
\Oo_\lambda:=[\iota_{\Delta*}\pi_\Delta^*\Oo_{\Bb^\heart}(\lambda)].
\]
\begin{lem}
\label{lem images of classical K-theory Bernstein generators}
We have
\[
\bar{\iota}^*\circ\bar{p}_*(\Oo_\lambda)=\Delta_*\Oo_{\Bb^\heart}(\lambda),
\]
and in the case when $G=\SL_2$, we have
\item 
\[
\bar{\iota}^*\circ\bar{p}_*(-q^{1/2}j_*\Oo(0,-2))=-q^{\frac{1}{2}}\Oo(0,-2)+q^{-\frac{1}{2}}\Oo(0,0).
\]
\end{lem}
\begin{proof}
By diagram \eqref{diagram image of C's in classical K-theory}
\[
\bar{\iota}^*\circ\bar{p}_*(-q^{1/2}j_*\Oo(0,-2))=-q^{\frac{1}{2}}\bar{\iota}^*\bar{\iota}_*\Oo(0,-2)=-q^{\frac{1}{2}}\lambda\otimes\Oo(0,-2)
\]
by \cite{CG}, Lemma 5.4.9, where $\lambda=[\Sym(\Pp^1\times T\Pp^1[1])]=[\Oo(0,0)]-q^{-1}[\Oo(0,2)]$. Here
we have multiplied $\Oo(0,2)$ by the character $q^{-1}$, giving its fibres the trivial $\C^\times$-action, which restores equivariance of the complex defining the class $\lambda$. (When confronted with a linear map $V\to W$ where $W$ has trivial $\C^\times$-action and $V$ is scaled by a character, one restores equivariance by tensoring $V$ with the inverse character.) Thus we have 
\[
\bar{\iota}^*\circ\bar{p}_*(-q^{1/2}j_*\Oo(0,-2))=-q^{\frac{1}{2}}\left(\Oo(0,0)-q^{-1}\Oo(0,2)\right)\otimes\Oo(0,-2)=
-q^{\frac{1}{2}}\Oo(0,-2)+q^{-\frac{1}{2}}\Oo(0,0).
\]
To prove the second formula, apply base-change diagram \eqref{diagram Bernstein subalgebra images classical K-theory} and use that, 
according to the Thom isomorphism theorem, $((\id\times\pi)^*)^{-1}=\bar{\iota}^*$. Then we
have
\[
\bar{\iota}^*\circ\bar{p}_*(\Oo_\lambda)=\bar{\iota}^*(\bar{p}\circ\iota_\Delta)_*\pi_\Delta^*\Oo(\lambda).
\]
By base-change, we have $(\id\times\pi)^*\Delta_*=(\bar{p}\circ\iota_\Delta)_*\pi_\Delta^*$,
hence $\Delta_*=((\id\times\pi)^*)^{-1}(\bar{p}\circ\iota_\Delta)_*\pi_\Delta^*$. 
\end{proof}
\begin{proof}[Proof of Proposition \ref{prop K-theory morphisms conjugate}]
That the bottom square commutes is the combination of the main results of \cite{XiIII} and \cite{Nie}.

By Proposition \ref{proposition J0 monoidal and ider pullback is monoidal}, the morphism $i_{\mathrm{der}}^*$ induces a morphism on $K$-theory as above. Hence 
by Lemma \ref{lem Phi is well-defined} and the above discussion, all the morphisms in the diagram are well-defined morphisms of algebras.

We first show the diagram commutes on the Bernstein subalgebra. 
Recalling from Section \ref{subsection Cat of Haff, Bez equiv} that the diagonal 
component of the Steinberg variety is a 
classical rather than a derived scheme, and so $\Oo_\lambda$ is 
naturally an element of $K_{0}(\St/\GGm)$ for which
$p_{\St*}\Oo_\lambda=\Oo_\lambda$ with the right-hand side regarded 
as an object of $K_{0}(\St^\heart/\GGm)$. Then by Lemma \ref{lem images of classical K-theory Bernstein generators}, it suffices to show that $\Phi([i^*\Oo_\lambda])=[\Delta_*\Oo_{\Bb^\heart}(\lambda)]$.
Indeed, though, we have
\[
i^*(\Oo_\lambda)=[\Oo_{\Delta\Bb}(\lambda)],
\]
as the structure sheaf of the diagonal pulls back to the structure sheaf 
of the diagonal. By Lemma \ref{lem Phi preserves structure sheaf of diagonal} we have 
$\Phi([\Oo_{\Delta\Bb}])=[\Oo_{\Delta\Bb^\heart}]$, and 
likewise for the twists. Thus the diagram commutes for the Bernstein subalgebra.

Consider the diagram
\begin{equation}
\begin{tikzcd}
\left(\Bb^\heart\times\Bb^\heart\times\Spec\left(\Sym\left(\g[1]\right)\right)\right)/\GGm\arrow[d, "r"]\arrow[r]&\left(\Bb^\heart\times\Bb^\heart\right)/\GGm\arrow[d, "j"]\\
(\Bb\times\Bb)/\GGm\arrow[d]\arrow[r, "i_{\mathrm{der}}"]&\St/\GGm\arrow[d]\\
\pt/\GGm\arrow[r]&\g/\GGm.
\end{tikzcd}
\end{equation}
Note that, by extending \eqref{diagram fibre product description of B x B}, the bottom square is Cartesian. Moreover, we have 
%
%
\begin{align*}
(\Bb^\heart\times\Bb^\heart)/\GGm\times_{\g/\GGm}\ptGGm &=
(\Bb^\heart\times\Bb^\heart)/\GGm\times_{\ptGGm}\ptGGm\times_{\g/\GGm}\pt/\GGm
\\
&=
(\Bb^\heart\times\Bb^\heart)/\GGm\times_{\ptGGm}\Spec\left(\Sym\left(\g[1]\right)\right)/\GGm
\end{align*}
where we identify $\g$ with $\g^*$ via the Killing form. Therefore the large square is also Cartesian,
and so the upper square is also Cartesian. 

We will use this to explain how to compute 
$i^*_{\mathrm{der}}j_*\Ff$  for any coherent sheaf $\Ff$ on $\Bb^\heart\times\Bb^\heart$.
Pullback by the top map sends $\Ff\mapsto\Ff\otimes\Sym(\g[1])$. By 
\eqref{eqn B defined by short exact sequence of vector bundles}, the map $p$ is given by the identity
on classical truncations, together with the map of cdgas which is pointwise the obvious map 
$\Sym(\bB[1])\to\Sym(\g[1])$ given by including a Borel subalgebra $\bB\into\g$. Pushforward by $r$
corresponds pointwise to equipping the module structure given by 
\[
\Sym(\bB_1[1])\otimes\Sym(\bB_2[1])\to\Sym(\g[1]).
\]
Applying the functor $(p\times\id)_*$ forgets the $\Sym[\bB_1[1])$-action, and then applying
$(\id\times i)^*$ quotients by the remaining $\Sym[\bB_2[1])$-action. That is, fibrewise we have
\[
\C\otimes_{\Sym(\bB_2[1])}\Sym(\g[1])\otimes\Ff_{\bB_1,\bB_2}=\Sym(\g/\bB_2[1])\otimes\Ff
=\Sym(\nN_2^*[1])\otimes\Ff_{\bB_1,\bB_2},
\]
where $\nN_2$ is the radical of $\bB_2$. 
Comparing with Lemma \ref{lem images of classical K-theory Bernstein generators} and the first
sentence of its proof, this says precisely that the diagram commutes when $G=\SL_2$.

Returning to general $G$, let $s$ be a simple reflection. 
Let $\bar{Y}_s\subset\Bb^\heart\times\Bb^\heart$ be the closure of the $G$-orbit labelled by $s$,
and let 
\[
\pi_s\colon T_{\bar{Y}_s}^*(\Bb^\heart\times\Bb^\heart)\to\bar{Y}_s
\]
be the conormal bundle. Put $\mathcal{Q}_s=\pi_s^*\Omega^1_{\bar{Y}_s/\Bb^\heart}$. Then by equation 
7.6.34 in \cite{CG}, it suffices to show that 
\[
\Phi\left(i^*i_{\St*}\mathcal{Q}_s\right)=\Oo_{\Bb^\heart}\boxtimes i_{X_s*}(q[\Oo(2)]-[\Oo]),
\]
where $i_{X_s}$ is the inclusion of the Schubert variety $X_s\simeq\Pp^1$ into $\Bb^\heart$. That is, 
the image of $\mathcal{Q}_s$ is just the pushforward of the answer in the $\SL_2$ case. But it is clear 
that this is indeed the case.

\end{proof}
We have now nearly proved
\begin{prop}
\label{prop multiplication of sheaves tee_w}
For $G=\SL_2$ or $\SL_3$, there exists a family of objects 
$\{\tee_w\}_{w\in\mathbf{c}_0}$ in $\J$, such that that if $t_wt_x=\sum_{z}\gamma_{w,x,z^{-1}}t_z$ in $J_0$, 
then
\[
\tee_w\star\tee_x=\bigoplus_z\tee_z^{\oplus\gamma_{w,x,z^{-1}}}
\]
in $\J$ and such that $\Phi([\tee_w])=[t_w]$.
\end{prop}
\begin{proof}
The discussion in Section \ref{subsubsection the sheaves tee_w} proves all but the last statement
of the proposition. Finally, by Proposition \ref{prop K-theory morphisms conjugate}, if $w=fw_0 g^{-1}$ we have
\[
\Phi([\tee_w])=\Phi([\Ff_f]\boxtimes[p^*\Gg_g])=[\Ff_f]\boxtimes i^*p^*[\Gg_g]=[\Ff_f]\boxtimes[\Gg_g].
\]
The general claim follows by linearity over $K_0(\ptG)$ and 
the parametrization in Section \ref{subsubsection the sheaves tee_w}.
\end{proof}
%
%
%
%
%
%
%
%
%
\subsection{The Schwartz space of the basic affine space}
\label{subsection The Schwartz space of the basic affine space}
Let $F$ be a non-archimedean local field and $G^\vee$ be a reductive 
group defined over $F$ dual to $G=G$.
The Schwartz space of the basic affine space $\mathcal{S}$ was defined by Braverman-Kazhdan in \cite{BKBasicAffine} to organize the principal series representations of $G^\vee(F)$ in a way insensitive to the 
poles of intertwining operators.
In \cite{BK}, Braverman and Kazhdan gave the following description of $J_0$ in terms of the Iwahori-invariants 
$\mathcal{S}^I$ of $\mathcal{S}$. In \textit{loc. cit.} it was shown 
that $\mathcal{S}^I$ is isomorphic to $K_{0}(\Bb/T\times\Gm^\heart)$ as an $H\otimes\C[\Waff]$-module,
and in \cite{BK}, it was proven that $J_0\simeq\End_{\Waff}(\mathcal{S}^I)$, where
the action of $\Waff$ is as defined in \textit{loc. cit.}
\begin{ex}
Let $G^\vee=\SL_2$, with $\Waff=W\ltimes X_*=\genrel{s_0,s_1}{s_0^2=s_1^2=1}$, where $X_*$ is the cocharacter lattice of $\G^\vee$, and $s_0$ is the finite 
simple reflection. Then we have that $K_{0}(\Pp^1/T\times\Gm)$ has basis $\{[\Oo_{\Pp^1}], [\Oo_{\Pp^1}(-1)]\}$, and we have $t_{s_0}=[\Oo_{\Pp^1}]\boxtimes[\Oo_{\Pp^1}]$ and
$t_{s_1}=[\Oo(-1)]\boxtimes[\Oo(-1)][1]$ under the identification in Lemma
\ref{lem CG map and Lusztig map conjugate}. The basis elements corresponding to 
the two distinguished involutions in $\cc_0$ act by projectors, with $t_{s_0}$
preserving $\Oo_{\Pp^1}$ and killing $\Oo_{\Pp^1}(-1)$, and vice-versa for $t_{s_1}$.
\end{ex}
Recalling that $K_{0}(\BbGGm^\heart\times_{\ptGGm}\BbGGm^\heart)\simeq K_{0}(\Bb/T\times\Gm^\heart)$, we see that
we have two natural coherent categorifications of $\mathcal{S}^I$, and that $\J$ acts on both of them:
\begin{prop}
\label{prop cat J acts on B B and Schwartz}
The category $\J$ acts on $\Coh(\BbGGm\times_{\ptGGm}\BbGGm)$ and on $\Coh(\Bb/T\times\Gm^\heart)$.
\end{prop}
\begin{proof}
This is a porism of Proposition \ref{proposition J0 monoidal and ider pullback is monoidal}; the 
proof that $\Ff\star\Gg$ is coherent if $\Ff,\Gg\in\J$ used nothing about $\sS(\Gg)$. 
The proof for $\Coh(\Bb/T\times\Gm^\heart)$ is entirely similar.
\end{proof}

\bibliography{J0_cat_paper_2021_biblio.bib}

\begin{bibdiv}
\begin{biblist}

\bib{AG}{article}{
      author={Arinkin, D.},
      author={Gaitsgory, D.},
       title={Singular support of coherent sheaves and the geometric
  {L}anglands conjecture},
        date={2015},
     journal={Sel. Math. New Ser.},
      volume={21},
       pages={1\ndash 199},
}

\bib{BDD}{article}{
      author={Bezrukavnikov, R.},
      author={Dawydiak, S.},
      author={Dobrovolska, G.},
       title={On the structure of the affine asympototic {H}ecke algebras,
  \textrm{with an appdendix by {R}. {B}ezrukavnikov , {A}. {B}raverman, {D}.
  {K}azhdan}},
        date={2023},
     journal={Transform. Groups},
}

\bib{tworel}{article}{
      author={Bezrukavnikov, R.},
       title={On two geometric realizations of an affine {H}ecke algebra},
        date={2016},
     journal={Publ. Math. Inst. Hautes \'{E}tudes Sci.},
      volume={123},
       pages={1\ndash 67},
}

\bib{BK}{inproceedings}{
      author={Braverman, A.},
      author={Kazhdan, D.},
       title={Remarks on the asymptotic {H}ecke algebra},
        date={2018},
   booktitle={Lie groups, geometry, and representation theory},
   publisher={Birkh\"{a}user, {C}ham},
       pages={91\ndash 108},
}

\bib{BKBasicAffine}{article}{
      author={Braverman, A.},
      author={Kazhdan, D.},
       title={On the {S}chwartz space of the basic affine space},
        date={1999Apr},
     journal={Sel. Math., New Ser.},
      volume={5},
      number={1},
       pages={1},
}

\bib{BKK}{article}{
      author={Bezrukavnikov, R.},
      author={Karpov, I.},
      author={Krylov, V.},
       title={A geometric realization of the affine asymptotic {H}ecke
  algebra},
        date={2023},
      eprint={arXiv:2312.10582},
}

\bib{BO}{incollection}{
      author={Bezrukavnikov, R.},
      author={Ostrik, V.},
       title={On tensor categories attached to cells in affine {W}eyl groups
  {II}},
        date={2004},
   booktitle={Representation {T}heory of {A}lgebraic {G}roups and {Q}uantum
  {G}roups},
      series={Advanced Studies in Pure Math.},
      volume={40},
       pages={101\ndash 119},
}

\bib{ChenDhillon}{article}{
      author={Chen, H.},
      author={Dhillon, G.},
       title={A {L}anglands dual realization of coherent sheaves on the
  nilpotent cone},
        date={2023},
      eprint={2310.10539},
         url={https://arxiv.org/abs/2310.10539},
}

\bib{CG}{book}{
      author={Chriss, N.},
      author={Ginzburg, V.},
       title={Representation theory and complex geometry},
   publisher={Birkh\"{a}user},
     address={Boston},
        date={1997},
}

\bib{D2}{article}{
      author={Dawydiak, S.},
       title={Denominators in {L}usztig's asymptotic {H}ecke algebra via the
  {P}lancherel formula},
        date={2021},
      eprint={arXiv: 2110.07148},
         url={https://arxiv.org/abs/2110.07148},
}

\bib{Rigid}{article}{
      author={Dawydiak, S.},
       title={The asymptotic {H}ecke algebra and rigidity},
        date={2023},
      eprint={2312.11092},
         url={https://arxiv.org/abs/2312.11092},
}

\bib{GRVol1}{book}{
      author={Gaitsgory, D.},
      author={Rozenblyum, N.},
       title={A {S}tudy in {D}erived {A}lgebraic {G}eometry vol. {I}},
      series={Math. Surveys Monogr.},
   publisher={Amer. Math. Soc.},
     address={Providence, RI},
        date={2017},
      volume={221},
}

\bib{KLDeligneLanglands}{article}{
      author={Kazhdan, D.},
      author={Lusztig, G.},
       title={Proof of the {D}eligne-{L}anglands conjectures for {H}ecke
  algebras},
        date={1987},
     journal={Invent. Math.},
      volume={87},
       pages={153\ndash 215},
}

\bib{affineII}{article}{
      author={Lusztig, G.},
       title={Cells in affine {W}eyl groups, {II}},
        date={1987},
     journal={J. Algebra},
      volume={109},
       pages={536\ndash 548},
}

\bib{LusCat}{article}{
      author={Lusztig, G.},
       title={Cells in affine {W}eyl groups and tensor categories},
        date={1997},
     journal={Adv. Math.},
      volume={129},
       pages={85\ndash 98},
}

\bib{Nie}{article}{
      author={Nie, S.},
       title={The convolution algebra structure on
  {$K^{G^\vee}(\mathcal{B}\times\mathcal{B})$}},
        date={2011},
      eprint={arXiv:1111.1868},
         url={https://arxiv.org/abs/1111.1868},
}

\bib{Propp}{thesis}{
      author={Propp, O.},
       title={A coherent categorification of the affine asymptotic {H}ecke
  algebra},
        type={Ph.D. Thesis},
 institution={MIT},
        date={2023},
}

\bib{QX}{article}{
      author={Qiu, Y.},
      author={Xi, N.},
       title={The based rings of two-sided cells in an affine {W}eyl group of
  type $\tilde{B}_3$, {I}},
        date={2022},
     journal={Sci. China Math.},
      volume={31},
       pages={65\ndash 81},
}

\bib{St}{article}{
      author={Steinberg, R.},
       title={On a theorem of {P}ittie},
        date={1975},
     journal={Topology},
      volume={14},
       pages={173\ndash 177},
}

\bib{ToenEMS}{article}{
      author={To\"{e}n, B.},
       title={Derived algebraic geometry},
        date={2014},
     journal={EMS Surv. Math. Sci.},
      volume={1},
      number={2},
       pages={153\ndash 240},
}

\bib{XiAMemoir}{book}{
      author={Xi, N.},
       title={The based ring of two-sided cells of affine {W}eyl groups of type
  $\tilde{A}_{n-1}$},
      series={Mem. Amer. Math. Soc.},
   publisher={Amer. Math. Soc.},
     address={Providence, RI.},
        date={2002},
      volume={157},
      number={749},
}

\bib{XiIII}{article}{
      author={Xi, N.},
       title={The based ring of the lowest two-sided cell of an affine {W}eyl
  group, {III}},
        date={2016},
     journal={Algebr. Represent. Theor.},
      volume={19},
       pages={1467\ndash 1475},
         url={https://doi.org/10.1007},
}

\bib{XiI}{article}{
      author={Xi, N.},
       title={The based ring of the lowest two-sided cell of an affine {W}eyl
  group},
        date={1990},
     journal={J. Algebra},
      volume={134},
       pages={356\ndash 368},
}

\end{biblist}
\end{bibdiv}

\end{document}